\documentclass[reqno]{amsart}
\usepackage{amssymb,amsmath,amsfonts,amsthm} 
\usepackage{enumitem} 
\usepackage{verbatim}
\usepackage{hyperref} 
\usepackage{graphicx}
\usepackage[capitalize]{cleveref} 
\usepackage{bbm}

\usepackage[numbers,sort&compress]{natbib}
\hypersetup{colorlinks=true,linkcolor=green,pdfborderstyle={/S/U/W 1}}

\usepackage{mathtools}  
\newcommand{\tmat}[1]{\begin{bsmallmatrix} #1 \end{bsmallmatrix}} 
\newcommand{\PWF}[4]{
  \left\{ \begin{array}{ll}
    #1 & \mbox{ if } #2 \\
    #3 & \mbox{ if } #4 \\
  \end{array}\right.}
\newcommand{\PWFO}[4]{
  \left\{ \begin{array}{ll}
    #1 & \mbox{ if } #2 \\
    #3 & \mbox{ otherwise #4}  \\
  \end{array}\right.}

\newcommand{\Z}{\mathbb{Z}}

\theoremstyle{plain}
\newtheorem{theorem}  {Theorem}  [section]
\newtheorem{lemma}  [theorem]   {Lemma}
\newtheorem{corollary}[theorem] {Corollary}
\newtheorem{construction} [theorem] {Construction}
\newtheorem{fact} [theorem] {Fact}
\newtheorem{proposition} [theorem] {Proposition}
\newtheorem{claim} [theorem] {Claim}
\newtheorem{claim*} {Claim}
\newtheorem{conjecture}[theorem] {Conjecture}
\newtheorem{quest}[theorem] {Question}
\newtheorem*{compprob}{Problem}

\theoremstyle{definition}

\newtheorem{definition}[theorem] {Definition}
\newtheorem{example}[theorem] {Example}

 \newcommand{\prob}[3]{\begin{compprob} #1 \\  
  {\bf Instance:} #2   \\
  {\bf Decision:} #3  
\end{compprob}}

\newcommand{\claimproof}{\renewcommand{\qedsymbol}{$\diamond$}}
\DeclareMathOperator{\Col}{Col}

\DeclareMathOperator{\Hom}{Hom}
\DeclareMathOperator{\NP}{NP}
\DeclareMathOperator{\Poly}{P}
\DeclareMathOperator{\PSPACE}{PSPACE}
\DeclareMathOperator{\Recol}{Recol}

\usepackage[colorinlistoftodos]{todonotes}

 \newcommand{\M}{M}
 \newcommand{\N}{N} 
 \DeclareMathOperator{\Mat}{\M}

 \newcommand{\g}{\mathbb{G}}
 \newcommand{\m}{\mathbb{M}}
 \renewcommand{\k}{\mathbb{K}}
 
 \DeclareMathOperator{\mHom}{\Hom_\m} 
 \DeclareMathOperator{\mCol}{\Col_\m} 
 \DeclareMathOperator{\mRecol}{\Recol_\m} 
 \DeclareMathOperator{\gHom}{\Hom_\g} 
 \DeclareMathOperator{\gCol}{\Col_\g} 
 \DeclareMathOperator{\gRecol}{\Recol_\g}

 \DeclareMathOperator{\kCol}{\Col_\k}
  
 \DeclareMathOperator{\kRecol}{\Recol_\k}
 \DeclareMathOperator{\PG}{PG} 
 
\usepackage{tikz}
\usetikzlibrary{arrows,positioning}
\tikzset{
  vert/.style={circle, draw=black!100,fill=black!100,thick, inner sep=0pt, minimum size=2mm},
  empty/.style={draw=none, fill=none, minimum size=0mm, inner sep=0pt}}
\newcommand{\Verts}[2][,]{\foreach \i/\x/\y in {#2}{\draw (\x#1\y) node (\i){};}}
\newcommand{\Edges}[1]{\foreach \i/\j in {#1}{\draw (\i) edge (\j);}}
\newcommand{\labEdges}[1]{\foreach \i/\j/\l/\p in {#1}{\draw (\i) edge node[empty, fill=white,  \p=.1pt]  {\l}   (\j);}}
 \newcommand{\Vlabel}[4][.3cm]{\draw node[empty, #3 of = #2, node distance = #1] () {$#4$};}
 \newcommand{\Vlabels}[2][.3cm]{\foreach \ver/\pos/\lab in {#2}{\Vlabel[#1]{\ver}{\pos}{\lab};}}
 \newcommand{\Alabel}[3]{\draw  (#1,#2) node[empty] () {$#3$};} 

 \newcommand{\mphi}{\phi_\m}

 \newcommand{\bone}[1]{\chi_{#1}} 
 \newcommand{\hcube}{\frac{1}{2}Q}
 \DeclareMathOperator{\id}{id}

\begin{document}
\thispagestyle{empty} 
\title{The complexity of matroid homomorphism reconfiguration}

\author{Cheolwon Heo}
\address{Applied Algebra and Optimization Research Center, Sungkyunkwan University, 2066 Seobu-ro, Suwon-si, Gyeonggi-do, South Korea, 16419}
\email{cwheo@skku.edu}

\author{Mark Siggers}
\address{Kyungpook National University Mathematics Department,
80 Dae-hak-ro, Daegu Buk-gu, South Korea, 41566}
\email{mhsiggers@knu.ac.kr}

\begin{abstract}
 We consider a reconfiguration version of the homomorphism problem ${\rm Hom}_\mathbb{M}(N)$ for binary matroids $N$. This reconfiguration problem,  ${\rm Recol}_\mathbb{M}(N)$,  asks, for two homomorphisms $\phi$ and $\psi$ of a matroid $M$ to $N$, if there is a path of homomorphism from $\phi$ to $\psi$ such that consecutive homomorphism in the path differ on a single cocircuit of $N$.  We show that this problem is trivial in the case that $N$ is, or dismantles to,
 the graphic matroid $M(K_2)$, and that the problem is $\PSPACE$-complete when $N$ is the graphic matroid $M(K_3)$, $M(K_4)$, or any graphic matroid containing $M(K_5)$. 
\end{abstract}

\maketitle

 \section{Introduction}

 All matroids in this paper are binary matroids, meaning that they have a representation by a matrix over $\Z_2$; and are simple, meaning the columns of the representing matrix are distinct.   
 A homomorphism from a matroid $\M$ to a matroid $\N$ is a point map that takes cycles to cycles.  
 More precisely, where the point space $P(\N)$ of a binary matroid is the vector space over $\Z_2$ with basis $E(\N)$, a point map $\tau: E(\M) \to E(\N)$ extends linearly to map $\tau$ of point spaces.  This means that for sets $X, Y \subset E(\M)$, $\tau(X \cup Y)$ is the symmetric difference of $\tau(X)$ and $\tau(Y)$. 

 \begin{definition}\label{def:hom}
 For binary matroids $\M$ and $\N$ a {\em matroid homomorphism} $\tau: \M \to \N$ is a point map $\tau:E(\M) \to E(\N)$ such that for every circuit $Z$ of $M$, $\tau(Z)$ is a disjoint union of circuits of $N$. We consider the empty set to be the disjoint union of an empty set of circuits.
 We write $\M \to \N$ to mean that there exists a matroid homomorphism $\tau: \M \to \N$. 
 \end{definition}

 Though there are several notions of matroid colourings, we will, in analogue to the graph case, refer to matroid homomorphisms of $\M$ to $\N$ as {\em $\N$-colourings of $\M$}. We refer to the points of $\N$ as {\em colours}.

 
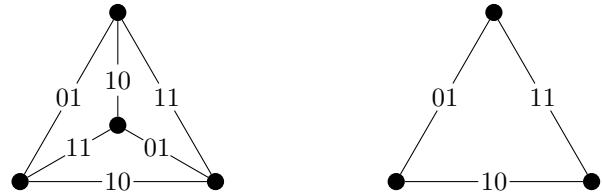
\begin{figure}[h!]
 \centering
 \begin{tikzpicture}[every node/.style={vert}]
 \begin{scope}[xshift = 0cm] 
 \Verts[:]{a/90/1.5, b/210/1.5, c/330/1.5, m/0/0}
 \Edges{a/b, c/m}
 \Edges{a/c, b/m}
 \Edges{b/c, a/m}
 \draw (30:.75) node[empty, fill=white] () {$11$};
 \draw (150:.75) node[empty, fill=white] () {$01$};
 \draw (270:.75) node[empty, fill=white] () {$10$};
 \draw (90:.6) node[empty, fill=white] () {$10$};
 \draw (210:.6) node[empty, fill=white] () {$11$};
 \draw (330:.6) node[empty, fill=white] () {$01$};
 \end{scope}
 \begin{scope}[xshift = 5cm] 
 \Verts[:]{a/90/1.5, b/210/1.5, c/330/1.5}
 \Edges{a/b}
 \Edges{a/c}
 \Edges{b/c}
 \draw (30:.75) node[empty, fill=white] () {$11$};
 \draw (150:.75) node[empty, fill=white] () {$01$};
 \draw (270:.75) node[empty, fill=white] () {$10$};
 \end{scope} 
 \end{tikzpicture}
 \end{figure}

 \begin{example}\label{K4K3}
  Never too early for a simple example, the figure above represents a matroid homomorphism of the graphic matroid $M(K_4)$ to the graphic matroid $M(K_3)$. The points of $M(K_3)$ are the non-zero points  $\{\tmat{0\\1},\tmat{1\\0},\tmat{1\\1} \}$ of $\Z_2^2$ (which we write compactly as $\{ 01, 10, 11 \}$). 

  These are the colours that we map the points of $M(K_4)$ to. The colours sum over any circuit, as vectors over $\Z_2$, to the zero vector $00$. So this represents a matroid homomorphism.  
 \end{example} 


 It is not difficult to see that a graph homomorphism $\phi: G \to H$ induces a matroid homomorphism $\mphi: \M(G) \to \M(H)$ of the corresponding graphic matroids by mapping a point (edge) $uv$ of $\M(G)$ to the point $\mphi(uv):= \phi(u)\phi(v)$ of $\M(H)$.  
 On the other hand, a matroid homomorphism $\M(G) \to \M(H)$ need not imply a graph homomorphism $G \to H$. 
 Indeed \Cref{K4K3} is a prime example of this, it gives a matroid homomorphism of $\M(K_4)$ to $\M(K_3)$, while clearly there is no graph homomorphism of $K_4$ to $K_3$. 

 There is, however, a familiar duality relating $\N$-colourings of a graphic matroid $\M(G)$ to $H$-colourings of the graph $G$.  One will notice that the condition that the colours on points of $\M(K_4)$ sum to $0$ on circuits is exactly the condition that makes this nowhere-$0$ colouring of the edges of $K_4$ a $\Z_2^2$-tension from Tutte's classic work in \cite{Tutte54}.   Indeed, for graphic matroids, our definition of a matroid homomorphism correspond with the notion of a {\em cyclic map} from \cite{LMT88}, a {\em $\Z_2$-tension continuous map} from \cite{DNR06} (this follows from Theorem 3.1 of their paper, and \Cref{compose} of ours), and a {\em cut continuous map} from \cite{NS08}. All these works build on \cite{Tutte54}.
 
 This basic setup of Tutte's duality is generalised by the notion of the decision graph, which we introduced in \cite{HKS}, only to notice later that it generalises similar constructions for graphic matroids from \cite{DNR06}.

 Defined in \Cref{sect:decisiongraphs}, the decision graph $D(\N)$, for a matroid $\N$, has the following property for any graph $G$,
 \[ G \to D(\N) \iff M(G) \to \N. \]
 
 In \cite{HKS} we used the decision graph to address the 
 complexity of the problem $\mHom(\N)$, for fixed $\N$, of deciding if a given matroid  $\M$ admits an $\N$-colouring. The complexity of the problem $\mHom(\N)$ for a matrix $\N$ was shown to be polynomial time solvable if
 $\N$ has a loop, or is bipartite, and was shown to be $\NP$-complete otherwise. This follows the dichotomy of the analogous result for graph homomorphisms from \cite{HN90}, and follows quickly from this result using the decision graph $D(\N)$ of a matroid $\N$. 

 For those unfamiliar with Tutte's duality, the following is an elementary example described using the decision graph. 

 \begin{example}\label{ex:Jaeger}
   The clique $K_4$ is the decision graph $D(\M(K_3))$ for the graphic matroid $\M(K_3)$ of $K_3$.  
   Indeed, using the representation of $\M(K_3)$ over $\Z_2$ by the matrix $A = \tmat{0 & 1 & 1\\  1 & 0 & 1}$ as in \Cref{K4K3}, label the vertices of $K_4$ with the columns, or points, $\{ 00, 01, 10, 11 \}$ of $\Z_2^2$.
   For a graph homomorphism $\phi: G \to K_4$ we get a 
   matroid homomorphism $\tau_\phi: \M(G) \to \M(K_3)$ by $\tau_\phi(uv) = \phi(u) + \phi(v)$.  On the other hand, for $\tau: \M(G) \to K_3$ 
   we get a graph homomorphism $\phi_\tau: G \to K_4$ by setting $\phi(v_0) = 00$ and then recursively defining $\phi_\tau$ on a vertex 
   $v$ of $G$ by setting $\phi_\tau(v) = \phi_\tau(v') + \tau(v'v)$ where $v'$ is the precursor of $v$ in a fixed spanning tree of $G$ rooted at $v_0$.   See \Cref{fig:ex1} for an example of a $K_3$-colouring of a graph $G$, the $\M(K_4)$-colouring $\tau_\phi$ it defines on $\M(G)$, and the $K_3$-colouring $\phi_{\tau_\phi}$ of $G$ that $\tau_\phi$ defines.  
 Notice that $\phi$ and $\phi_{\tau_\phi}$ differ by the constant $11$.  
 \end{example}
 
  \begin{figure}
    \renewcommand{\tmat}[1]{#1}
    \begin{tikzpicture}[x=1.5cm, y=1.5cm, every node/.style = {vert}]
  
       \Verts{1/0/0, 2/0/1, 3/1/1, 4/1.7/.5, 5/1/0}
       \Edges{1/2, 2/3, 3/4, 4/5, 5/1, 3/5}
       \Vlabels[.4cm]{1/below/\tmat{10},2/above/\tmat{11},3/above/\tmat{10}, 4/right/\tmat{11},5/below/\tmat{01}}
       \Alabel{.5}{.5}{\phi}

     \begin{scope}[xshift=4cm]
       \Verts{1/0/0, 2/0/1, 3/1/1, 4/1.7/.5, 5/1/0}
       \labEdges{1/2/$\tmat{01}$/, 2/3/$\tmat{01}$/, 3/4/$\tmat{01}$/, 4/5/$\tmat{10}$/, 5/1/$\tmat{11}$/, 3/5/$\tmat{11}$/}
       \Alabel{.5}{.5}{\tau_\phi}
     \end{scope}

     \begin{scope}[xshift=8cm]
       \Verts{1/0/0, 2/0/1, 3/1/1, 4/1.7/.5, 5/1/0}
       \Edges{1/2, 2/3, 3/4, 4/5, 5/1, 3/5}
       \Vlabels[.4cm]{1/below/\tmat{01},2/above/\tmat{00},3/above/\tmat{01}, 4/right/\tmat{00},5/below/\tmat{10}}
       \Alabel{.5}{.5}{\phi_{\tau_\phi}}
     \end{scope}

\end{tikzpicture}
  \caption{A $K_4$-colouring $\phi$ of a graph $G$, the $\M(K_3)$-colouring $\tau_\phi$ of $\M(G)$ that it defines, and the $K_4$-colouring $\phi_{\tau_\phi}$ of $G$ that $\tau_\phi$ defines.}\label{fig:ex1}
  \end{figure}
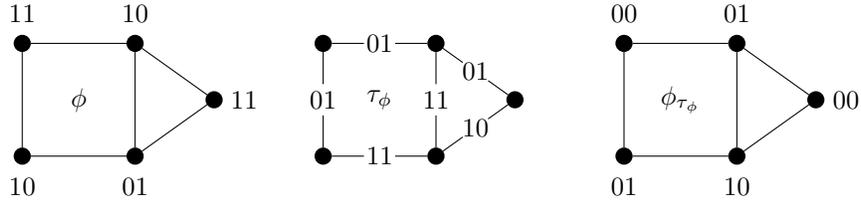


 Retraction and list variations of the problem $\Hom_\M(N)$ were also considered in \cite{HKS}; for these variations, while the decision graph was still the main tool, there were several differences between the graph and matroid versions of the problem.  In this paper we look a reconfiguration version of $\Hom_\M(H)$

 In a reconfiguration version of a decision problem, one does not look to decide if there is a solution to an instance, but looks to decide if there is series of small changes one can make to get between a given pair of solutions. One well known reconfiguration version of the graph colouring problem $\gHom(H)$ is the graph recolouring problem $\gRecol(H)$. Where the colouring graph $\gCol(G,H)$ is the graph on he set of $H$-colourings of $G$ in which two are adjacent if they differ on a single vertex, one must decide, for given maps $\phi$ and $\psi$, if there is a path between $\phi$ and $\psi$ in $\gCol(G,H)$. 
 
 The colouring graph generally has order exponential in the order of the instance $G$, so deciding if such a path exists can take exponential time. It is not hard to show that the problem $\gRecol(H)$ is in the complexity class $\PSPACE$. 
 It was shown in \cite{CvdHJ08} that the problem $\gRecol(K_n)$
 is in $\Poly$ if $n \leq 3$ and in \cite{BonCer09} that it is $\PSPACE$-complete if $n \geq 4$.
 Here is our reconfiguration graph.

 \begin{definition}\label{def:matrecgraph}
 The {\em matroid $\N$-colouring graph $\mCol(\M,\N)$ of $\M$} for binary matroids $\M$ and $\N$, is the graph whose vertex set is the set of matroid homomorphisms from $\M$ to $\N$, in which 
 two vertices $\tau$ and $\tau'$ are adjacent if for some cocircuit $C$ of $\M$, $\tau(e) \neq \tau'(e)$ if and only if $e \in C$. 
 \end{definition}

 \begin{example}
  A homomorphism $\tau : M(H) \to \M(K_3)$ of the graphic matroid of $H$ assigns the colours $01,10$ and $11$ to the edges of $H$, and as these must sum to $0$ over $\Z_2$ on any cycle, the parity of the number of occurrences of each colour, on a cycle, must be the same. For the graph $H = C_5$, for example, a homomorphism $\tau: M(C_5) \to M(K_3)$ must have three edges of one colour, and one edge of each of the other colours. So there are $60$ vertices in 
  $\mCol(M(C_5),M(K_3))$. Any two edges of $C_5$ are a cut-set, so we can switch the colours on any two edges, and get an adjacent colouring. This allows us to switch which colour appears three times, and to reorder the colours. So $\mCol(\M(C_5),\M(K_3))$ is connected. 
 \end{example}

 This seems to be a natural definition of recolouring as if the image of a single point $e$ under a homomorphism is changed, 
 at least one other point in every cycle containing $e$ must also be changed. A minimal set of points such that the intersection with every circuit has even cardinality is cocircuit. In Section \ref{sect:tech} we give an equivalent, and often more convenient, definition using a matrix representation of $\N$ to define addition: we show that $\tau \sim \tau'$ in $\mCol(\M,\N)$ if and only if there is a constant $c$ and a cocircuit $C$ of $\M$ such that \[ \tau' = \PWFO{\tau(e) + c}{e \in C}{\tau(e)}{.} \]

 The {\em matroid $\N$-recolouring problem} is defined as follows. 

\prob{$\mRecol(\N)$}{A binary matroid $\M$ and two maps $\tau,\sigma \in \mCol(G,H)$.}{Is there a path in $\mCol(G,H)$ between $\tau$ and $\sigma$?}

 This problem is clearly in the complexity class $\PSPACE$. We expect that it will be polynomial time solvable, or $\PSPACE$-complete, depending on $\N$. 
 Our main theorem is the following, which gives the complexity of $\mRecol(\N)$ in the case that $\N$ is the graphic matroid of a clique $K_n$. We say that $\mRecol(\N)$ is {\em trivial} if all instances are  
 YES instances.

 \begin{theorem}\label{FullMain}
  The problem $\mRecol(\M(K_n))$ is trivial if $n \leq 2$, and is otherwise $\PSPACE$-complete. 
 \end{theorem}

  As is true in the graph case, proving the tractable half of this theorem in easy, though for graphs the tractable cases are not all trivial in the sense that we've defined it, as $\gCol(G,K_2)$ is either empty or has two components for any instance $G$. 

  \begin{fact}\label{FullMainEasy}
      The problems $\mRecol(\M(K_1))$ and $\mRecol(\M(K_2))$ are trivial. 
  \end{fact}
  \begin{proof}
      That $\mRecol(\M(K_1))$ is trivial, is easy to see.   
      As an instance of $\mRecol(M(K_2))$ consists of vertices of $\mCol(\M,M(K_2))$ for some $\M$. But this is empty if $\M$ has an odd cycle, and contains a single vertex otherwise. Either way, 
      all instances are YES instances. 
  \end{proof}


 Our main goal is therefore proving that $\Recol(\M(K_n))$ is $\PSPACE$-complete for $n \geq 3$. The proof does not proceed as one might expect, being easier in the case that $n$ is a power of $2$. In proving out main result, we prove several other results.  We now give a brief outline of the paper and the proof, mentioning some of the other results that we get along the way.

 In Section \ref{sect:tech} we give alternate versions of Definitions \ref{def:hom} and \ref{def:matrecgraph} that depend on a particular 
 binary representation, and then provide some useful tools for working with these definitions. 

 In Section \ref{sect:Induced} we give matroid analogues of several useful `categorical-type' graph results, showing that homomorphisms 
 of matroids often induce homomorphisms of matroid recolouring graphs, or of their transitive closures.  In particular, we define a matroid dismantling, and prove, in \Cref{smallcliques}, that if $\N$ dismantles to $M(K_2)$ (or to a loop) then $\mRecol(\N)$ is trivial. 
 This is a partial analogue of a result of Brightwell and Winkler from \cite{BW00} which shows that the graph recolouring problem $\gRecol(H)$ is trivial if and only $H$ is dismantlable.  
 Conjecture \ref{conj:dismant} posits that $\mRecol(\N)$ is trivial if and only if $\N$ dismantles to $\M(K_2)$.   

In Section \ref{sect:decisiongraphs} we recall/generalise the definition of the decision graph, and use it to prove \Cref{FullMain}
in the case that $n$ is a power of $2$ by reducing the problem of Kempe $K_4$-recolouring to matroid recolouring. In fact \Cref{HardK2t} says that for $t \geq 2$,  $\mRecol(\M(K_{2^t}))$ is not only $\PSPACE$-complete, but is $\PSPACE$-complete for graphic instances. This gives the $n = 3$ case of \Cref{FullMain} for graphic instances too, using results from Section \ref{sect:Induced}. 

Finally, in \Cref{sect:alln} we finish off the proof of \Cref{FullMain} by showing, in \Cref{thm:otherKn}, that $\mRecol(\M(G))$ is $\PSPACE$-complete for any graph $G$ containing a $K_n$ for $n \geq 5$.  We expect that this holds for graphic instances too, but we were unable to show it.

 \section{Technical Defintions}\label{sect:tech}

 Recall that all matroids in this paper are simple binary matroids. 
 A {\em loop} in a matroid is a circuit consisting of a singular point.
 A matroid with a loop is called {\em looped}.  In Sections \ref{sect:tech} and \ref{sect:Induced}, results and definitions are for matroids which may have loops. From Section \ref{sect:decisiongraphs}
 we consider only loopless matroids.

 Any simple binary matroid $\N$ can be representated by an  $m \times |E(\N)|$ matrix $A$ over $\Z_2$ with distinct columns. The value $m$ can be taken to be the rank of $\N$, but sometimes it is useful to use a greater $m$.
 The columns of $A$ are the points of $\N$, and as they are vectors in space $\Z_2^m$ they can be added together.  Their span in $\Z_2^m$ is the {\em point space} $P(\N)$ of $\N$.  A subset of columns is a cycle if it sums, over $\Z_2$ to $0$.  

  The requirement, for a map $\tau: E(\N) \to E(\M)$, that $\tau(Z)$ is a disjoint union of circuits for any circuit $Z$ becomes the requirement that if $\sum_{x \in Z} x = 0$, then   $\sum_{x \in Z} \tau(x) = 0$; that is, it becomes the requirement that $\tau$ is a linear map of $P(\N)$ to $P(\M)$. Though the sum of two elements depends on the representation $A$, the fact that a set of points sums to zero does not depend on $A$, and so the following equivalent definitions of a matroid homomorphisms, which use binary representations of the matroids, are independent of the actual representations used.  
 
 \begin{fact} For binary matroids $\M$ and $\N$ a map $\tau: E(\M) \to E(\N)$ is a matroid homomorphism if and only if it defines a linear map of point spaces,  if and only if for all circuits $Z$ of $\M$ 
  \[ \tau(Z) := \sum_{e \in Z} \tau(e)  = 0. \] 
 \end{fact}

Recall from \Cref{def:matrecgraph} that two homomorphisms in $\mCol(\M,\N)$
are adjacent if and only if they differ on a cocircuit of $C$ of $\M$.  The following   shows that if this happens, they must differ on that cocircuit on a constant.  We note that Proposition 4.1 of \cite{DNR06} gives the following two facts in the case that the matroid $\N$ is a vector space.

 \begin{fact}\label{ConstOnCocirc}
 Let $\tau, \tau': \M \to \N$ be matroid homomorphisms for binary matroids $\M, \N$, and let $C$ be a cocircuit of $\M$.
 Then the following are equivalent.
 \begin{enumerate}
  \item For $e\in E(\M)$, $\tau(e) \neq \tau'(e)$ if and only if $e \in C$.
  \item There exists a non-zero vector $c \in P(\N)$ such that 
 \[ \tau' = \tau + c \chi_C:= \PWFO{\tau(e) + c}{e \in C}{\tau(e)}{.} \]
 \end{enumerate}
 \end{fact}
 \begin{proof}
 It is easy to see that (2) implies (1) as $c$ is a non-zero vector.
 For the converse, let us assume (1).
 Let $\tau: E(M) \to P(\N)$ be given by $\tau(e):=\tau(e)-\tau'(e)$.
 Then (1) implies that $\tau(e) \neq 0$ if and only if $e \in C$.
 Now, let us assume that there exist points $e_1,e_2 \in C$ such that $\tau(e_1) \neq \tau(e_2)$.
 Without loss of generality, we may assume the first entries of $\tau(e_1)$, and $\tau(e_2)$ as columns in $P(\N)$  are $1$ and $0$, respectively.
 Let $X$ be the set of all $e \in C$ such that the first entry of $\tau(e)$ is $1$.
 As $\tau$ and $\tau'$ are matroid homomorphisms from $\M$ to $\N$, for any circuit $Z$ of $\M$,
 \begin{align*}
  \Sigma_{e \in C \cap Z} \tau(e) &= \Sigma_{e \in Z} \tau(e)\\
  &= \Sigma_{e \in Z} (\tau(e) - \tau'(e)) \\ &= \tau(Z) - \tau'(Z) = 0 - 0 \\ &= 0 
 \end{align*}
 This implies that $|X \cap Z|$ is even. 
 Our choice of $Z$ is arbitrary, so $X$ is a disjoint union of cocircuits of $\M$.
 Since $X$ is non-empty and $X \subseteq C$, we have $X=C$, contradicting the fact that $\tau(e_2) \in C-X$.
 \end{proof}

 So two maps are adjacent in $\mCol(\M,\N)$ if and only if they differ by a constant on a cocircuit.
 The following complements this, saying that we can get a new homomorphism from an old one by adding a constant on a cocircuit as long as doing so is well defined. 

 \begin{fact}\label{fact:buildhom}
  Let $\M = \Mat(A)$ be a binary matroid, $\tau: \M \to \N$ be a matroid homomorphism, $C$ be a cocircuit of $\M$, and $c \in P(\N)$ be such that $\tau(e) + c$ is in $E(\N)$      for every $e \in C$. The map $\tau' = \tau + c\bone{C}$ 
   is also a matroid homomorphism, and is adjacent to $\tau$ in $\mCol(\M,\N)$. 
 \end{fact}
 \begin{proof}
  This is trivial as for every circuit $Z$ of $\M$, $|Z \cap C|$ is even.
 \end{proof} 

 Again, the $c$ in the notation $\tau + c\bone{C}$ depends on the representation $A$ of $\N$, but the fact that there is such a $c$ defining the edge between two maps in $\mCol(\M,\N)$ does not.   We say that we get $\tau+c\bone{C}$ from $\tau$ by {\em adding $c$ to $\tau$ on $C$}. 


 \begin{example}\label{example:PG} 
  Recall that for $t \geq 0$ the matroid represented by the matrix over $\Z_2^t$ consisting of all $2^t$ possible columns is called the {\em looped projective geometry} $PG^\ell(t,2)$, or is just called a vector space. 
 
 Where $PG = PG^\ell(t,2)$ for any $t \geq 0$ we show that $\mCol(\M,PG)$ is connected for all $\M$, and so $\mRecol(PG)$ is trivial. Indeed, for any $\tau$ and $\sigma$ in $\mCol(\M,PG)$ there is a path between them of length at most the number of elements of a basis $T$ of $PG$ on which $\tau$ and $\sigma$ differ. 

  The statement is clear if $t = 0$, so assume that $t \geq 1$. Pick any basis $T$ of $\M$ and order its elements: $b_1, \dots, b_n$. Let $\tau_0 = \tau$ and assume that we have a path from $\tau_0$ to a vertex $\tau_{i-1}$ that agrees with $\tau$ on $\{b_1, \dots, b_{i-1}\}$. The set $C_i$ of elements $f$ of $\M \setminus \{f\}$ that such that $T \cup \{f\}$ contains a circuit $Z_f$ containing $b_i$, is a cocircuit of $\M$. The map \[ \tau_i = \tau_{i-1} + (\sigma(e)-\tau_{i-1}(e))\bone{C_i}, \] which is a homomorphism by \Cref{fact:buildhom}, 
  is adjacent to $\tau_{i-1}$ and agrees with $\sigma$ on one more element of $T$. 
 By induction, we get a path to the homomorphism $\tau_n$ that agrees with $\sigma$ on all elements of $T$, so is $\sigma$. 
 \end{example}

 \section{Induced homomorphsims, retractions, and dismantling}\label{sect:Induced}

 Recall that a problem $A$ {\em polynomially reduces} to a problem $B$ if for an instance $I_A$ of $A$, we can construct an instance $I_B$ of $B$ that is of size polynomial in that of $I_A$, such that $I_A$ is a YES instance of $A$ if and only if $I_B$ is a YES instance of $B$. In this case we write $A \leq_{\Poly} B$. We say, informally, that $B$ is at least as hard as $A$, which means, in particular, that $A$ is polynomially time solvable if $B$ is, and $B$ is $\PSPACE$-hard if $A$ is. If $A \leq_{\Poly} B \leq_{\Poly} A$ then the problems are {\em polynomially equivalent} and we write $A =_{\Poly} B$.

 A basic fact in the theory of graph recolourings which comes from \cite{BW00} is the fact that a homomorphism $f: H \to H'$ induces a homomorphism $f^*: \gCol(G, H) \to \gCol(G,H')$.  If the homorphism is a retraction, it induces a polynomial reduction $\gRecol(H) \leq_{\Poly} \gRecol(H')$, and if it is a dismantling retraction it induces a polynomial equivalence of $\gRecol(H) =_{\Poly} \gRecol(H')$. We observe here that the same holds for matroid homomorphisms.


 For matroid homomorphisms $\alpha:  L \to M$ and $\beta: \M \to \N$ we get, of course, that $\beta \circ \alpha: L \to \N$ is a matroid homomorphism. 
 This extends naturally to walks in the colouring graphs in a couple of ways.

 \begin{lemma}\label{compose}

 Let $\alpha:  L \to M$ and $\beta: \M \to \N$  be matroid homomorphisms.  
  \begin{enumerate}
       \item If there is a walk $A$ between $\alpha$ and $\alpha'$ in $\mCol(L,\M)$, then there is a walk between $\beta \circ \alpha$ and $\beta \circ \alpha'$ in $\mCol(L,\N)$.
               We call this walk $\beta \circ A$.
       \item If there is a walk $B$ between $\beta$ and $\beta'$ in $\mCol(\M,\N)$, then there is a walk  between $\beta \circ \alpha$ and $\beta' \circ \alpha$ in $\mCol(L,\N)$.                  We call this walk $B \circ \alpha$.
    
   \end{enumerate}
                              
 \end{lemma} 
 \begin{proof}
  For the first statement, it is enough to consider the case that $\alpha \sim \alpha'$. We write 
  $\alpha' = \alpha + c\bone{C}$ for some constant $c \in E(\M)$ and some cocircuit $C$ of $L$.
  To be pedantic, we write this as 
      \[ \alpha'(e) = \alpha(e) + \PWF{c}{e \in C}{0}{e \not\in C}. \] 
  By the linearity of $\beta$ we have 
   \[ [\beta \circ \alpha'](e) = [\beta \circ \alpha](e) + \PWF{\beta(c)}{e \in C}{0}{e \not\in C},\]
  which we can write as $\beta \circ \alpha' = \beta \circ \alpha + \beta(c)\bone{C}$. Thus $\beta \circ \alpha \sim \beta \circ \alpha'$ in $\mCol(L, \N)$ as needed. 

  For the second statement, we write $\beta' = \beta + c \bone{C}$, which allows us to write 
  \[ [\beta' \circ \alpha](e) = [\beta \circ \alpha](e) + \PWF{c}{\alpha(e) \in C}{0}{\alpha(e) \not\in C},\]
  and so $\beta' \circ \alpha = (\beta \circ \alpha) + c\bone{\alpha^{-1}(C)}$. As the pre-image $\alpha^{-1}(C)$ of a cocircuit $C$ is a cocycle, so a disjoint union of cocircuits, 
  we get a path from $\beta' \circ \alpha$ to $\beta \circ \alpha$ by recolouring one cocircuit at at time, as needed. 
 \end{proof}

 In particular, part (1) not only takes paths to paths, but takes edges to edges, so  gives us the following. 

 \begin{fact}\label{fact:dualhom}
 For a matroid homomorphism $f: \N \to \N'$, and a matroid $\M$, the induced map $f^*: \mCol(\M, \N) \to \mCol(\M,\N')$ defined by $f^*(\tau) = f \circ \tau$ 
 is a graph homomorphism. 
 \end{fact}

 A surjective homomorphism $r: \M \to \M'$ is a {\em retraction} if there is a homomorphism $c: \M' \to \M$ such that $r \circ c$ is the identity on $\M'$. If this exists then $c$ is an isomorphism, and identifying $\M'$ with $c(\M')$ we can view $\M'$ as a submatroid of $\M$ that is fixed by $r$. A matroid $\M'$ is a {\em retract} of $\M$ if there is a retraction $r: \M \to \M'$. 

 \begin{proposition}\label{prop:MRetEquiv}
  If a matroid $N'$ is a retract of $N$, then there is a path between two vertices of $\mCol(\M,\N')$ if and only if there is a path between them in $\mCol(\M,\N)$. In particular $\mRecol(N') \leq_{\Poly} \mRecol(N)$.
 \end{proposition}
 \begin{proof}
  A path between vertices of $\mCol(\M,\N')$ is a path between them in $\mCol(\M,\N)$. On the other hand, by \Cref{fact:dualhom}, a path between them in $\mCol(\M,\N)$ maps, by $r^*$ where $r: \N \to \N'$ is the retraction, to a path between them in $\mCol(\M,\N')$.
 \end{proof}

  A retraction $r \in \mCol(\N,\N)$ is a {\em dismantling retraction} if it is adjacent to the identity map in $\mCol(\N,\N)$. 
  Recall that a dismantling retraction $r: H \to H$ of graphs is any retraction whose image $H'$ has one fewer vertices than $H$. These induce dismantling retractions on the graphic matroids, so our definition can be viewed as a generalisation of the definition of a dismantling retraction for graphs.

  \begin{proposition}\label{prop:dismant}
   If $r: \N \to \N'$ is a dismantling retraction, then there is a path between two vertices $\tau$ and $\tau'$ of $\mCol(\M,\N)$ if and only if there is a path between $r \circ \tau$ and $r \circ \tau'$ in $\mCol(\M,\N')$. In particular, $\mRecol(\N') \equiv_{\Poly} \mRecol(\N)$.   
  \end{proposition}
  \begin{proof}
  By \Cref{fact:dualhom}, we have that a path between $\tau$ and $\tau'$ in $\mCol(\M,\N)$ maps by $r^*$ to a path between $r \circ \tau$ and $r \circ \tau'$ in $\mCol(\M,\N')$. 
  
  To see that a path between $r \circ \tau$ and $r \circ \tau'$
  gives a path between $\tau$ and $\tau'$ it is enough to show, for any $\tau \in \mCol(\M,\N)$, that there is a path from $\tau$ to $r \circ \tau$ in $\mCol(\M,\N)$. But this is immediate from part (2) of \Cref{compose} from the fact that $\id \sim r$ in $\mCol(\N, \N)$. 
  \end{proof}

  If there is a series of dismantling retractions whose composition takes a matroid $\N$ to a retract $\N'$ then we say that $\N$ dismantles to $\N'$. 
  
 \begin{corollary}\label{smallcliques}
  If $\N$ dismantles to the loop $M^\ell(K_1)$ or the edge $M(K_2)$ then 
  $\mRecol(\N)$ is trivial. 
 \end{corollary}
 \begin{proof}
  By \Cref{prop:dismant} it is enough to show 
  $\mRecol(M^\ell(K_1))$ and $\mRecol(M(K_2))$ are trivial, and the latter is done in \Cref{FullMainEasy}.  For the former, observe that for any instance $\M$ the graph $\mCol(\M, M^\ell(K_1))$ is a single vertex, and so all instances of $\mRecol(M^\ell(K_1))$ are YES instances. 
 \end{proof}

 In the graph case, if follows by a result of \cite{BW00} that $\gRecol(H)$ is trivial if and only if $H$ is dismantlable. We wonder if an analogue holds for matroids.

\begin{conjecture}\label{conj:dismant}
   $\mRecol(N)$ is trivial if and only if $N$ dismantles to the loop $\M^\ell(K_1)$ or the edge $\M(K_2)$.
\end{conjecture}   
 
 \subsection{Applications}

 Having proved some basics of the matroid recolouring graph, we get some quick results.

 \begin{example}
  Where $\PG^\ell(t-1,2)$ is represented as the matrix of all columns over $\Z_2$ of length $t$, the projection map onto the first $t-1$ coordinates is linear, so is a retraction of $\PG^\ell(t-1,2)$ onto $\PG^\ell(t-2,2)$. Indeed it is 
  $\id + e_t \chi_C$ where $e_t$ is the $t^{th}$ normal vector, and $C$ is the cutset of elements have a $1$ in the $t^{th}$ coordinate. So it is a dismantling retraction. 
 \end{example}

 The looped projective geometries $\PG^\ell(t-1,2)$ for 
 $t = 0,1,2$ are the looped cliques $\M^\ell(K_n)$ for $n = 1,2,3$ respectively. 
 
 \begin{example}\label{Dismantk4k3}
  Where $\M(K_4)$ is represented as the set of columns of length $4$ over $\Z_2$ that have $2$ entries, the map in \Cref{K4K3} can be defined as the map $\id + [1,1,1,1]^T \chi_C$ where $C$ is the set of points with a $1$ in the fourth coordinate. This is adjacent to the identity, so is a dismantling retraction, showing that $\M(K_4)$ dismantles to $\M(K_3)$.
 \end{example}

 \begin{example}\label{Dismantknl}
   All reflexive cliques dismantle to the loop, so $\M(K_n^\ell)$ dismantles to $\M(K_1^\ell)$ for all $n \geq 2$.  Indeed where 
   $\M(K^\ell_n)$ is represented as the set of columns of length $n$ over $\Z_2$ that have $0$ or $2$ entries, the map $\id + [0, \dots, 0,1,1]^T\chi_C$ where $C$ is the set of points with a $1$ in the $n^{th}$ coordinate is a dimantling retraction of $\M(K_n^\ell)$ to 
   a submatroid isomorphic to $\M(K_{n-1}^\ell)$.
   \end{example}

  This reproves \Cref{FullMainEasy} in a more systematic way, and gives
  us a looped version as well. 
  
 \begin{fact}\label{fact:tract}
  The problem $\mRecol(\M(K_n))$ is trivial if $n \leq 2$, and the problem $\mRecol(\M^\ell(K_n))$ is trivial for all $n$. 
 \end{fact}

 Now we turn our attention to the hardness part of \Cref{FullMain}, we start with the main tools we need for this.

 \section{Decision Graphs}\label{sect:decisiongraphs}



 A {\em decision} graph for a  matroid $N$ is a graph $D$ such that for any graph $G$,
  \[ G \to D \iff \M(G) \to N. \]
 These were constructed for all non-bipartite binary matroids in \cite{HKS} and used to show that $\mHom(N)$ is hard for all $N$ with girth $3$.  
 
 Our construction did not use the ideas of Tutte's duality, which we recalled in the introduction, but a much simpler construction using these ideas was (essentially) given in Proposition 3.4 of \cite{DNR06}. To translate the construction to matroids, one requires a representation, and different representations can give different (though homomorphically equivalent) graphs.   We give this construction without a representation, for all binary matroids, and then observe that it is exactly the construction from \cite{DNR06} in the case that the representing matrix is an $m \times n$ matrix where $m$ is the rank of $N$.

 \begin{definition}
  Let $N$ be a binary matroid of rank $m$ and size $n$, and let $B$ be a basis of $N$. The {\em decision graph} $D(N,B)$ is a graph on the power set of $B$. Again viewing the power-set as a $\Z_2$-module, two sets $V,W \subset B$ are adjacent if $|V \triangle W| =1$ or there exists $e \in E-B$ such that $(V \triangle W) \cup \{e\}$ is a circuit. 
 \end{definition}

 This is implicit to the matroid, but using a representation, has a form that is generally easier to use. Letting $N$ be represented by a matrix $A$ where the elements $B$ of the basis are the columns with a single non-zero entry, the graph defined in the following definition is clearly $D(N,B)$.   
 
 \begin{definition}Let $A$ be an matrix representing $N$ over $\Z_2$.The decision graph $D(N,A)$ is the graph on $\Z_2^m$ such that columns $v$ and $w$ of $\Z_2^m$ are adjacent if $v + w$ is a column of $A$. 
 \end{definition}

 If $N$ is a matrix of rank $r$ but is represented by a matrix with more than $r$ rows, then $D(N,A)$ has more vertices than $D(N,B)$.
 But if $A$ has $r$ rows, and the columns corresponding to the elements of $B$ make an identity matrix, then these are the same graph, independent of representation. 

 Indeed, it follows by the unique representability of binary matroids, proved in \cite{BL76}, that any two representations
 of a rank $r$ binary matroid $\N$ by matrices over $\Z_2$ with $r$ rows differ by elementary row operations, so where one is $A$, the other is $A'A$ for some invertible $r \times r$ matrix $A'$.
 The vertex map $\phi: D(N,A) \to D(N,A'A)$ defined by $\phi(v) = A'v$ is a graph isomorphism, as the invertibility of $A'$ implies that $v+w = e$ if and only if $A'v + A'w = A'e$.

 \begin{fact}
  For a binary matroid $N$ of rank $r$ and size $n$,5 the decision graph $D(N,B)$, for any basis $B$, is isomorphic to $D(N,A)$ for any representation of $N$ by an $n$-by-$r$ matrix $A$
  over $\Z_2$. 
 \end{fact}

 It is shown in Proposition 3.4 of \cite{DNR06} that the graphs $D(N,A)$ for representations $A$ of $D$ are decision graphs.  As $D(N,B)$ is the smallest of these, we now denote it by $D_u(N)$ and call it {\bf the} decision graph of $N$. We mentioned before that $K_4$ is the decision graph for $M(K_3)$. The details of this are given in the following example, as $M(K_3)$ is the projective geometry $\PG(1,2)$. 

 \begin{example}\label{ex:DGofClique}
 
  The projective geometry matroid $PG(t-1,2)$ has representation as the matrix of all non-zero columns of $\Z_2^t$. Its universal decision graph is $K_{2^t}$. 
  
  The matroid clique $M(K_n)$ has representation as the matrix of columns of $\Z_2^{n-1}$ with weight $1$ or $2$. Its universal decision graph is the $n$-dimensional half-cube $\hcube_n$-- the vertices are the columns of $\Z_2^{n-1}$ and two are adjacent if they differ in one or two coordinates.

 \end{example}

 \subsection{The decision graph and the induced Tutte connection}  

 The proof from Proposition 3.4 of \cite{DNR06} that $D_u(N)$ is a decision graph uses the Tutte construction that we saw in \Cref{ex:Jaeger} to get  an element of $\gCol(G, D_u(N)$ for every element of $\mCol(\M(G),N)$, and vice versa. 

\begin{definition}[The Tutte Connection]\label{def:Jaeger}
 For a graph $G$, fix a vertex $v_0$ and a spanning subtree $T$ of $G$ rooted at $v_0$.  This defines a partial ordering of $V(G)$ by saying that $u \leq v$ if $u$ is in the path in $T$ from $v$ to $v_0$. Let $v_0,v_1, \dots, v_{n-1}$ be an
ordering of $V(G)$ that refines this partial order.  Let $b$ be an vertex  of $D_u(N)$. 
\begin{itemize}
\item 
For any $\tau \in \mCol(\M(G),\N)$ 
  let $\phi_{\tau,b}: V(G) \to V(D_u(N))$ be defined by setting 
 \begin{equation}\label{Jaeger}
     \phi_{\tau,b}(v_0) = b, \qquad \mbox{and} \qquad  \phi_{\tau,b}(v_j) = \phi_{\tau,b}(v_{i}) + \tau(v_iv_j)
\end{equation}
 for $j = 1, \dots, n-1$, where $v_i$ is the neighbour of $v_j$ in $T$ of lower index. 
\item  For a graph homomorphism $\phi: G \to D_u(N)$, define the map $\tau_\phi: M(G) \to N$ by 
 \begin{equation}\label{reverseJaeger}
  \tau_\phi(uv) = \phi(u) + \phi(v).
 \end{equation}
 \end{itemize}
 \end{definition}
 
 The standard proof of Tutte shows that $\phi_{\tau,b}: G \to D_u(N)$ and $\tau_\phi: \M(G) \to N$ are homomorphisms (see  Proposition 3.4 of \cite{DNR06})
 and that $\tau_{\phi_{\tau,b}} = \tau$.  Thus on the level of vertices, the decision graph induces maps between $\mCol(\M(G),N)$  and $\gCol(G,D_u(\N))$.
 Letting $\Phi_b(\tau) = \phi_{\tau,e}$ and $\Phi^{-1}(\phi) = \tau_\phi$, we have that for each $b \in V(D_u(\N))$, 
         \[  \Phi_b: V(\mCol(M(G),N)) \to V(\gCol(G,D_u(N))) \]
 is an injection with inverse $\Phi^{-1}$. 
 Let $\gCol(G,D_u(N);v_0,b)$ be the subgraph of $\gCol(G,D_u(N))$ induced on the   the vertices $\phi$ for which $\phi(v_0) = b$.

 \begin{fact}
 For a matroid $N$, and a graph $G$ with designated vertex $v_0$ of $G$, the map
 \[ \Phi_{b}:   V(\mCol(M(G),N)) \to  V(\gCol(G,D_u(N);v_0,b)) \]
  is a bijection with inverse $\Phi^{-1}$.   
 \end{fact}

 As $b$ runs over the vertices of $D_u(\N)$, the graphs $\gCol(G,D_u(N);v_0,b)$ partition $\gCol(G,D_u(N))$.  Moreover, as $D_u(\N)$ is vertex transitive, they induce isomorphic subgraphs.  Because of this we generally drop the subscript $b$, assuming it is the $0$ vector in $V(D_u(N))$ and let $\Phi = \Phi_0$.  Similarily, the choice of $T$ in \Cref{def:Jaeger} is unimportant, and changing $v_0$ only changes $b$ in $\phi_{\tau,b}$ to another constant, so we generally remove $v_0$ from the notation and write
$\gCol(G,D_u(\N);0)$ for $\gCol(G,D_u(N);v_0,b)$.

 Now while $\Phi$ and $\Phi^{-1}$ give a nice $1$-to-$|D_u(N)|$ correspondence between the vertex sets of $\mCol(M(G),N)$ and $\gCol(G,D_u(\N))$, neither of them induce a homomorphism.  Indeed, the following example shows that $\Phi$ preserves no connectivity at all. 

 \begin{example}
 The decision graph of $\M(K_3)$ is $D_u(\M(K_3)) = K_4$.
 While the recolouring graph $\gRecol(K_4,K_4)$ is easily seen to consist of $24$ independent vertices, $\gRecol(\M(K_4),\M(K_3))$ is a $K_{3,3}$-- from a given colouring, we can recolour by switching the colours on any two colour classes. 
 \end{example}

  On the other hand, one can show that $\Phi^{-1}$ maps components of $\gCol(G,D_u(N))$ to components of $\mCol(M(G),N)$.  
  We make this into a bijection of components by replacing $\gCol(G,D_u(N))$ by a super-graph that we get by considering Kempe recolourings.

 \subsection{Matroid recolouring to Kempe recolouring}

 Inspired by Kempe's flawed proof of the four-colour theorem, the notion of a (vertex) Kempe recolouring of a graph has appeared 
 in many contexts, and indeed the notion of  edge Kempe recolourings, which one gets via the Tutte connection, have also appeared. 
 In \cite{KempeRecon20} the reconfiguration problem was considered for Kempe recolourings.  
 Given a map $\phi \in \gCol(G,K_n)$ a component of the subgraph induced on the set $\phi^{-1}(a,b)$ of vertices of two colours $a$ and $b$ is 
 an $(a,b)$-component of $G$. A Kempe recolouring of $G$ is  when one toggles the colours $a$ and $b$ on an $(a,b)$-component.

  The following was proved in \cite{KempeRecon20}.
  
  \begin{theorem}\label{kemperecol}
   The problem $\kRecol(K_n)$ of deciding, for two $n$-colourings of a graph $G$, if there is a path between them in $\kCol(G,K_n)$, is $\PSPACE$-complete for $n \geq 4$.  
  \end{theorem}

 Replacing $K_n$ with any decision graph, we can give the following reframing of this definition, which allows us to `add' colours.
 It is clear that this coincides with the above notion of `toggling' when  $D$ is a clique $K_{2^t} = D_u(PG(t-1,2))$.   

 \begin{definition}
  Let $N$ be a matroid and $D = D_u(N)$ be its decision graph. 
  For a graph $G$, maps $\phi$ and $\psi$ in $\gCol(G,D)$ are {\em Kempe-adjacent}, or, adjacent in $\kCol(G,D)$ if 
   \[ \psi = \phi + b\bone{U} \]
  for some $b,b'$ in $[n]$ and some component $U$ of $\phi^{-1}(b',b'+b)$ in $G$. 
  \end{definition}

  For an edge $\phi \sim \phi'$ of $\gCol(G,D_u(N))$, $\phi'$ and $\phi$ differ on only one vertex $v$, and where $b = \phi'(v) - \phi(v)$ the vertex $v$ is a component of $\phi^{-1}(\phi'(v), \phi'(v) + b)$ as no neighbour of $v$ can have either of these colours. So $\gCol(G,D_u(N))$ is a subgraph of $\kCol(G,D_u(N))$. 
  The following says that $\Phi^{-1}$  maps components of $\kCol(G,D_u(N))$ to  components of $\mCol(M(G),N)$. Clearly it does the 
  same for the components of the subgraph $\gCol(G,D_u(N))$. 

 \begin{lemma}\label{easyMK}
   For any matroid $N$ and graph $G$, if $\phi \sim \phi'$ in $\kCol(G,D_u(N))$, then there is a path from $\tau_\phi = \Phi^{-1}(\phi)$ to $\tau_{\phi'} = \Phi^{-1}(\phi')$ in $\mCol(M(G),N)$.
 \end{lemma}
 \begin{proof}

  Assume that $\phi \sim \psi \in \kCol(G,K_{n})$.
 So, by definition,  $\psi = \phi + b\bone{U}$ where $U$ is a $(b,b+b')$-component of $G$ for some $b'$. Let $E = \delta(U)$ be the edges of $G$ from $U$ to $\bar{U}$. This is a cutset of $G$ so is a disjoint union of cocycles $C_1, \dots, C_d$. We have that 
  $\tau_{\phi'} = \tau_\phi + b\bone{C}$, and so we get a path from $\tau_\phi$ to $\tau_{\phi'}$ in $\mCol(M(G),\N)$ by adding $b$ to the cocycles $C_i$, one at a time.

 \end{proof}

  In general $\Phi = \Phi_0$ does not even take edges of $\mCol(\M(G),N)$ to paths of $\kCol(G,D_u(N))$.  But it does when
   we restrict our attention to the case that $D$ is a clique $K_{2^t} = D_u(PG(t-1,2))$.

  \begin{theorem}\label{thm:MK}
  Let $\N = PG(t-1,2)$ for $t \geq 3$ so $D_u(\N) = K_{n}$ with $n = 2^t \geq 4$. There is path from $\tau_\phi$ to $\tau_\psi$ in $\mCol(M(G),N)$ if and only if there is a path from $\phi = \Phi_0(\tau_\phi)$ to $\psi=\Phi_0(\tau_\psi)$ in $\kCol(G,K_{n})$. 
  \end{theorem}
  \begin{proof} \Cref{easyMK} gives us the easy direction.  For the other direction,  let $\tau_\phi$ and $\tau_\psi$ be adjacent. 
   We show that there is, in fact, a path from  $\phi$ to $\psi$ in $\kCol(G,K_{n};0)$. 
 We have $\tau_\psi = \tau_\phi + b\bone{C}$ for some cocycle $C$ of $G$. Let $U \subset V(G)$ be such that $\delta(U) = C$, and $v_0 \in \bar{U}$. 
  We claim that for any vertex $u \in U$, all neighbours of colour $\phi(u) + b$ are also in $U$, and so $U$ can be partitioned into 
  $(b',b+b')$-components $U_i$ of $G$ for various $b'$. (Such components cannot intersect, as addition is over $\Z_2$, so pairs any $x$ uniquely with $x + b$.)
  Indeed, let $u$ be in $U$, $v$ be a neighbour in $\bar{U}$, and assume towards contradiction that $\phi(v) = \phi(u) + b$. Then $\tau(uv) = b$ and so, as $uv$ is in $C$, we have $\sigma(uv) = 0$, which is not allowed.
  Starting at $\phi$ and adding $b\bone{U_i}$ for each $i$ we get a path in $\kCol(G,K_n;0)$ to $\psi$. 
  \end{proof}

Thus we can solve $\kCol(G,K_{2^t})$ with $\mCol(M(G),P(t-1,2))$. The reduction is clearly polynomial. Thus using \Cref{kemperecol}, we get the following. 

  \begin{corollary}\label{cor:PGhard}
   The problem $\mRecol(P(t-1,2))$ is $\PSPACE$-complete for $t \geq 3$. Indeed, it is $\PSPACE$-complete for graphic instances. 
  \end{corollary}

 Now consider the following example. 
 
 \begin{example}\label{ex:matret}
  Recall from \Cref{ex:DGofClique} that the matroid clique $M(K_n)$ has representation as the matrix of columns of $\Z_2^{n-1}$ with weight $1$ or $2$. The projective geometry matroid $PG(t-1,2)$ has representation as the matrix of all non-zero columns of $\Z_2^t$. 
  
  Let $n = 2^t$ and let $B$ be the set of columns in $\Z_2^{2^t-1}$ of weight $1$. This set has cardinality $|B| = 2^t-1$ so there exists a bijection $\beta_0$ to the non-zero columns of $\Z_2^t$. As $B$ is a basis of $\Z_2^{2^t-1}$ this extends uniquely to a linear map $\beta: \Z_2^{2^t-1}$ to $\Z_2^t$. As this is linear, it restricts, on $M(K_{2^t})$ to a matroid homomorphism to $PG(t-1,2)$. 
  The inverse $c = \beta_0^{-1}$ of $\beta_0$ is also a matroid homomorphism of $PG(t-1,2)$ into $M(K_{2^t})$ as it is also linear and all elements of $B$ are points of $M(K_n)$. As $r \circ c$ is the identity, $r$ is a retraction of $M(K_{2^t})$ to $PG(t-1,2)$. 
 \end{example}

 This example shows that there is a retraction from $\M(K_{2^t})$ to $PG(t-1,2)$. So by \Cref{prop:MRetEquiv}, we get the following $n = 2^t$ cases of \Cref{FullMain} immediately from \Cref{cor:PGhard}.
 
  \begin{corollary}\label{HardK2t}
   The problem $\mRecol(\M(K_{2^t}))$ is $\PSPACE$-complete for $t \geq 2$. Indeed it is $\PSPACE$-complete for graphic instances. 
  \end{corollary}

  This shows in particular that $\mRecol(\M(K_4))$ is $\PSPACE$-complete.  In \Cref{Dismantk4k3} we saw that $\M(K_4)$ dismantles to $\M(K_3)$ and so by \Cref{prop:dismant} we get the following $n = 3$ case of \Cref{FullMain}. 

  \begin{corollary}\label{HardK3}
     The problem $\mRecol(\M(K_{3}))$ is $\PSPACE$-complete for graphic instances.
  \end{corollary}

  We cannot imagine that the same result  would not hold for all cliques $K_n$ with $n \geq 4$, but we can only show it using a non-graphic reduction. We do this in the following section. But first, a question.

  We wonder for which matroids $N$ one can prove Theorem \ref{thm:MK}. 
  \begin{quest}
   For which $N$ it is true that there is a path between $\tau$ and $\tau'$ in $\mCol(M(G),N)$  if and only if there is a path between $\Phi_0(\tau)$ and $\Phi_0(\tau')$ in $\kCol(G,D_u(N))$?
  \end{quest}
   
  Necessarily by \Cref{easyMK}, for such $N$ we have that there is a path between $\tau$ and $\tau'$   in $\kCol(G,D_u(N);0)$ if and only if there is a path between them in $\kCol(G,D_u(N))$.  This is easy to show for $K_n$ for $n \geq 3$ using automorphisms of $K_n$, indeed an automorphism $\alpha$ of $D_u(N)$ switching only two vertices $0$ and $v$ can be used to show  that any $\phi \in \kCol(G,D_u(N);0)$ is adjacent in $\kCol(G,D_u(N))$
to $\alpha\circ \phi \in \kCol(G,D_u(N);v)$.

\section{Matroid $n$-recolouring is \rm{PSPACE}-complete}\label{sect:alln}

 In this section, we prove \Cref{thm:otherKn} which extends  \Cref{HardK2t} to $K_n$ for $n \geq 5$.  With \Cref{FullMainEasy}, this gives \Cref{FullMain}.

\newcommand{\V}[1]{V_{#1}}
\newcommand{\Vn}{V_n}
\newcommand{\VN}{V_N}
\newcommand{\U}[1]{U_{#1}}
\newcommand{\Un}{U_n}
\newcommand{\En}{E_n}
\newcommand{\E}[1]{E_{#1}}
\newcommand{\Nk}{\N_K}
\newcommand{\Nkf}{\N_{K_4}}
\newcommand{\NZ}{\N_Z}
\newcommand{\Mk}{\M_K}
\newcommand{\Mkf}{\M_{K_4}}
\newcommand{\MZ}{\M_Z}

 Recall that the set $\Vn = \{e_1, \dots, e_n \}$ of standard normal vectors in $\Z_2^n$ is a basis of $\Z_2^n$.  The {\em support} of a vector in $\Z_2^n$ is the set of elements in this basis used to represent it, its {\em weight} is the number of vectors in its support. Let $E_n$ be the set of weight two vectors of $\Z_2^n$.  

  The standard {\em graphic representation} of a graphic matriod $\M(G)$ is the representation by the matrix $A_G$ with columns  $E(\M(G)) = E_G = \{ v_i + v_j \mid v_i \sim_G v_j \}$ in $\Z_2^n$ where the vertex set  $V = V(G) = \{v_1, \dots, v_n\}$ is viewed as the set $\Vn$.  
  In the case of the matroid $M(K_n)$, $E_{K_n} = \{ v_i + v_j \mid v_i \neq v_j \in V \}$; we often write this as $2V$.

  We start with a simple and believable fact. 

    \begin{fact}\label{k5auto}
      For $n \geq 5$ and any loopless graph $G$, any homomorphism $\tau: M(K_n) \to M(G)$ maps $M(K_n)$ isomorphically to a copy of $M(K_n)$ in $M(G)$.  
    \end{fact}
    \begin{proof}
     Let $\M(K_n)$ have a graphic representation with $V(K_n) = \{u_0, \dots, u_{n-1}\}$, and let $\M(G)$ have a graphic representation by $A_G$     with $V(G) = \{v_0, \dots, v_{m-1} \}$.  

    Any homomorphism $\tau: \M(K_n) \to M(G)$ is $\Phi^{-1}(\phi)$ for some homomorphism $\phi: K_n \to D$ of $K_n$ to the decision graph $D = D(\M(G),A_G)$ for which $\phi(u_0) = 0$. As adjacent vertices in $D$ differ by columns in $A_G$, which have weight $2$ over $V(G)$, 
   a clique of size $n \geq 5$ in $D$ containing the vertex $0$ can be assumed to be
  $\{0, v_0 + v_1, v_0 + v_2, \dots, v_0 + v_{n-1}\}$ by permuting the labels in $V(G)$. So we may assume that $\phi(u_i) = v_0 + v_i$ for all $i$. It follows that  $\tau(u_i + u_j) = v_i + v_j$ for all $i,j \in \{0, \dots, n-1\}$ and so  $\tau$ maps $M(K_n)$ isomorphically to a copy of $M(K_n)$ in $M(G)$, as needed. Indeed $\tau(u_0 + u_j) = 0 + u_0 + u_j = u_0 + u_j$, and when $i,j \neq 0$ we have 
\[ \tau(u_0 + u_j) = \phi(u_i) + \phi(u_j) = v_0 + v_i + v_0 + v_j = v_i + v_j. \]  
    \end{proof}

 \begin{theorem}\label{thm:otherKn}
   Let $\N$ be a loopless graphic matroid containing a copy of $M(K_n)$ for $n \geq 5$. The problem $\mRecol(\N)$ is $\PSPACE$-complete. 
 \end{theorem}
\begin{proof}

    Let $\N$ be a graphic matroid containing a copy of $M(K_5)$. Let $n \geq 5$ be the largest $n$ such that 
   $\N$ contains a copy of $M(K_n)$. We use the graphic representation of $\N$ over $V_\N = \{v_1, \dots v_N\}$.
   We assume that the span of $V_n = \{v_1, \dots, v_n\}$ induces a copy $\Nk$ of $\M(K_n)$, and let $\Nkf$ 
   be the copy of $\M(K_4)$ induced by the span of $V_4 = \{v_1, \dots, v_4 \}$.   Let $\NZ$ be the $4$-star 
   $ \{v_1 + v_n, v_2 + v_n, v_3 + v_n, v_4 + v_n\}$ in $\Nk$.
   
   We use a gadget construction to  reduce $\mRecol(M(K_4))$ to $\mRecol(\N)$, showing that the latter is hard.

 \begin{construction}
  For a matroid $\M$, construct the matroid $\M^*$ as follows. 
  Let $\Mk$, graphically represented over $U_n = \{u_1, \dots, u_n\}$, be a copy of $\Nk$. 
 Let $\Mkf$ be the copy of $\Nkf$ on $U_4 = \{u_1, \dots, u_n\}$ and $\MZ =  \{u_1 + u_n, u_2 + u_n, u_3 + u_n, u_4 + u_n\}$ be the copy of $\NZ$.   
Construct $M^*$ from $M$ and $\Mk$ by adding, for each $e \in E(\M)$, the new point $e'$ and the circuit 
 $\{e, e'\} \cup \MZ$.   Where $E$ is the  point set of $\M$ let $E' = \{e' \mid e \in E\}$;  so $\M^*$ has point set $E \cup E' \cup 2U_n$.  
 \end{construction}

 Observe that there is a basis of $\M^*$ consisting of a basis of $M$ and the star
 $\{u_i+u_n \mid u_i \in \Un \setminus \{u_n\}\}$ which contains  $\MZ$. With respect to this basis, the fundamental cycle containing a point $e' \in E'$ is 
 $C_{e'} = \{e, e',\} \cup \MZ$.
   Every homomorphism $\tau \in \mCol(\M, M(K_4))$ can be viewed as a homomorphism  $\tau: \M \to \N_{K_4} \subset \N$.  We extend such a homomorphism $\tau$ to a map $s(\tau)$ of $M^*$ to $\N$ by letting it take $\Mk$ identically to $\Nk$; that is, by setting $s(\tau)(u_i + u_j) = v_i + v_j$.
   As this  defines $s(\tau)$ on our basis of $M^*$, and so it extends linearly to a map $s(\tau)$ from $E(\M^*)$ to $\Nk \subset \N$.  

 \begin{claim}
 For a homomorphism $\tau: \M \to \M(K_4)$, the map $s(\tau): \M^* \to \N$ is a homomorphism. 
 \end{claim}
 \begin{proof}\claimproof
 The map $s(\tau)$ is homomorphism on the submatroids $\M$ and $\Mk$ of $\M^*$, which contain a basis of $\M^*$, so extending linearly, this is a homomorphism   as long as for any $e' \in E'$ the linear extension  $s(\tau)(e') = s(\tau)(e) + s(\tau)(\NZ) = s(\tau) + v_1 + v_2 + v_3 + v_4$ is indeed an edge of $\N$. 
  But $s(\tau)(e)$ is in $2\V4$ and so this is also in $2\V4$,  and so is an edge of $\Nkf \subset \N$.
  \end{proof}

  On the other hand, we have the following. 

 \begin{claim}
   For a homomorphism $\sigma: M^* \to \N$, the restriction $\sigma_M$ of $\sigma$ to $M$ is a homomorphsim of $M$ to some copy of $\M(K_4)$ in $\N$  
  \end{claim}   
  \begin{proof}\claimproof
      By definition $\M$ contains a copy $\Mk$ of $\M(K_n)$ which by \Cref{k5auto} maps isomorphically to a copy of $\M(K_n)$ in $\N$,
      and so we may assume that $\sigma(u_i + u_j) = v_i + v_j$ for all $u_i + u_j \in 2U_n = E(\Mk)$. Thus $\sigma(\MZ) = v_1 + v_2 + v_3 + v_4$.   As $0 = \sigma(C_{e'}) = \sigma(e) + \sigma(e') + \sigma(\MZ)$ we get that $\sigma(e) + \sigma(e') = v_1 + v_2 + v_3 + v_4$, and so $\sigma(e)$ and $\sigma(e')$, being in $2\VN$,  are both in $2\V4 \subset E(\N)$, as needed.

 \end{proof}

To finish the proof we show that  there is a path from $\tau$ to $\tau'$ in $\mCol(\M, \M(K_4))$ if and only if there is a path from $s(\tau)$ to $s(\tau')$ in  $\mCol(\M^*, \N)$.   As $\mRecol(\M(K_4))$ is $\PSPACE$-complete, this is enough to show that $\mRecol(\N)$ is too.


We use the graphic representation of  $\M(K_4)$ over $E_4 = \{e_1, e_2, e_3, e_4\}$. The first direction is quite easy.

\begin{lemma}
  If $\tau \sim \tau'$ in $\mCol(\M, \M(K_4))$ then there is a path from 
  $s(\tau)$ to $s(\tau')$ in  $\mCol(\M^*, \N)$.  
\end{lemma}  
\begin{proof} 
 Assume that $\tau \sim \tau$ in $\mCol(\M,\M(K_4))$.  So $\tau' = \tau + (e_i + e_j)\chi_C$ for some $i,j \in [4]$ and some cocircuit $C$ of 
  $\M$.  We claim that $C \cup C'$ where $C' = \{e' \mid e \in C\}$ is a cocycle of $\M^*$. Certainly it has even intersection with every cycle of $M$, and empty, so even, intersection with every cycle of $\Mk$.  Its intersection with the fundamental cycles $C_{e'}$ are also even, and so as we have verified it on a basis of the cycle space, its intersection with every cycle is even.  Thus $C \cup C'$ is indeed a cocycle of $\M^*$. It is then the disjoint union of cocircuits $C_1, \dots, C_r$ and so by \Cref{fact:buildhom} we get from $s(\tau)$ to $s(\tau')$ by adding $e_i + e_j$ to these cocircuits, one at a time.    
  \end{proof}

 The other direction is a bit more work.  

\begin{lemma}
   If there is a path from $s(\tau)$ and $s(\tau')$ in $\mCol(\M^*,\N)$, then there is a path from $\tau$ to $\tau'$ in $\mCol(\M, \M(K_4))$.   
\end{lemma}
\begin{proof}
  Let $s(\tau) = \sigma_0 \sim \sigma_1 \dots \sim \sigma_n = s(\tau')$ be a path
  from $s(\tau)$ to $s(\tau')$ in $\mCol(\M^*, \N)$.  We show, for an edge $\sigma \sim \sigma'$ in $\mCol(\M^*, \N)$, that there is a path from $\sigma_M$ to $\sigma'_M$ in $\mCol(\M, \M(K_4))$.  As clearly $s(\tau)_M = \tau$ for any $\tau$, we have that 
  $(\sigma_0)_M =  \tau$ and $(\sigma_n)_M = \tau'$, so this is enough.

    Let  $\sigma \sim \sigma'$ in $\mCol(\M^*, \N)$. We find a path from $\sigma_M$ to $\sigma'_M$. 
    As by \Cref{k5auto} $\sigma_M$ restricted to $\Mk$ maps isomorphically to some copy of $M(K_n)$ in $\N$, we can 
    again assume that  $\sigma_M(u_i + u_j) = v_i + v_j$ for $u_i + u_j \in 2\Un$.   
    As $\sigma \sim \sigma'$ we can write $\sigma' = \sigma + c \chi_C$ for some cocircuit $C$ of $\M^*$, which restricts on $\Mk$ to a cocircuit $C_K$.

    \begin{claim*}
      We may assume one of the following.
      \begin{enumerate}
       \item  $C_K$ is  $\{u_1\} \times \{u_2, \dots, u_n\}$ and so $c = v_1 + v'_1$ where the span of $\{v'_1, v_2, \dots, v_n\}$ induces
              another copy of $M(K_n)$ in $\N$, or
       \item $C_K$ is  $\{u_1, u_2\} \times \{u_3, u_4, \dots, u_n\}$, and so $c = v_1 + v_2$.
       \end{enumerate}
     \end{claim*}
    \begin{proof}\claimproof
        As $\Mk$ is graphic, the cocircuit $C_K$ is $U + \bar{U} = \{ u_i + u_j \mid u_i \in U, u_j \in \bar{U} \}$ for 
        some subset $U$ of $\Un$, and so $\sigma_M$ maps its points to colours $\{ v_i + v_j \mid    u_i \in U, u_j \in \bar{U} \}$.
        Adding $c \in 2\VN$ to each of these colours, we must get colours in $2\VN$, but this is only possible, upto permuting $[n]$,
       in the two listed cases. 

    \end{proof}

    In case (2), $\sigma'_\M = \sigma_\M + (v_1 + v_2)\chi_{C_K}$, and so $\sigma_\M \sim \sigma'_\M$ in $\mCol(\M,\M(K_4))$.  
    In case (1), we claim that a point $e$ of $\M$ is in $C_K$  if and only if 
    $v_1$ is in the support of $\sigma_M(e)$.  Indeed as $\sigma'_M$ maps $\Mk$ to a clique on  $\{v'_1,v_2, \dots, v_n\}$, $v_1$ is not in the support of any $\sigma_M'(e)$, which implies this.  Thus $\sigma'_\M$ is just $\alpha \circ \sigma_\M$ for the isomorphism $\alpha$ between the two copies of $M(K_n)$ that is induced by switching $v_1$ and $v_1'$.  It is a simple thing to check that there is an edge from the identity map $\id$ to the transposition $\alpha$ in $\mCol(\M(K_4),\M(K_4)$, and so by part (2) of \Cref{compose} there is a walk in $\mCol(\M, \M(K_4))$ from $\id \circ \sigma_M = \sigma_M$ to $a \circ \sigma_M = \sigma'_M$, as needed.

   \end{proof}
   \end{proof}

  \section{Conclusion}

  In proving \Cref{FullMain} we showed that $\mRecol(\M(G))$ is $\PSPACE$-complete for $G$ being $K_3$ or $K_4$ or any graph that contains a $K_5$. We expect the answer to one of the following questions to be true.

 \begin{quest} 
   Is $\mRecol(\M(G))$ $\PSPACE$-complete for any graph $G$ containing a $K_3$?  Or failing this, for any graph containing a $K_4$.\end{quest}    

 We have that this is true in the case that $G$ dismantles to $K_3$ or $K_4$, but this is a much stronger condition than containing $K_3$ or $K_4$. 

 We have not done any thing for binary matroids $\N$ that are not graphic, except that some of these may dismantle to a loop, and so have tractable recolouring problems.  These matroids do have recolouring graphs though, and we expect that the following might be true. 

 \begin{quest}
  Is $\mRecol(\N)$ $\PSPACE$-complete for any matroid $\N$ for which $D_u(\N)$ is loopless and contains a $K_4$?
 \end{quest}

  The only tractabililty results we have are when the problem is trivial. An example of a problem that is tractable but non-trivial would be interesting. In the graph case, it is shown in \cite{Wroch15} that $\gRecol(G)$ is tractable if $G$ has girth at least $5$; the problem is clearly non-trivial. What about for matroids?

\begin{quest}
  Is $\mRecol(\M(C_5))$ tractable? Is $\mRecol(\M(G))$ tractable   for all $G$ of girth at least $5$?  
\end{quest}

\newcommand{\doi}[1]{\href{http://dx.doi.org/#1}{\small\nolinkurl{DOI: #1}}}
\renewcommand{\url}[1]{\href{https://arxiv.org/abs/#1}{\small\nolinkurl{arXiv: #1}}}


\end{document}